\documentclass{amsart}
\usepackage{amsmath,amssymb,amsthm,amsfonts,amscd,mathrsfs,mathtools,enumerate}
\usepackage{tikz,graphicx}
\usepackage[all]{xy}
\usepackage[hyphens]{url}
\usepackage{hyperref}
\usepackage{caption}
\usetikzlibrary{decorations.markings}

\theoremstyle{plain}
    \newtheorem{thm}{Theorem}[section]
    \newtheorem{lem}[thm]   {Lemma}

    \newtheorem*{thm*}{Theorem}

\theoremstyle{definition}
    \newtheorem{defn}[thm]  {Definition}

    \newtheorem{rems}[thm]{Remarks}

\newcommand{\be}{\begin{enumerate}}
\newcommand{\ee}{\end{enumerate}}
\newcommand{\bi}{\begin{itemize}}
\newcommand{\ei}{\end{itemize}}
\newcommand{\R}{\mathbb{R}}

\newcommand{\TC}{{\sf TC}}
\newcommand{\dTC}{\overrightarrow{{\sf TC}}}
\newcommand{\dPX}{\overrightarrow{P}X}
\newcommand{\dsn}{\overrightarrow{S^n}}

\begin{document}

\title[Directed TC of spheres]{Directed topological complexity of spheres}

\author{Ay\c{s}e Borat}

\author{Mark Grant}

\address{Bursa Technical University,
Faculty of Engineering and Natural Sciences,
Department of Mathematics,
Yildirim/Bursa,
Turkey}

\address{Institute of Mathematics,
Fraser Noble Building,
University of Aberdeen,
Aberdeen AB24 3UE,
UK}

\email{ayse.borat@btu.edu.tr}

\email{mark.grant@abdn.ac.uk}

\date{\today}


\begin{abstract}
We show that the directed topological complexity (as defined by E.\ Goubault \cite{Goubault}) of the directed $n$-sphere is $2$, for all $n\ge1$.
\end{abstract}


\maketitle
\section{Introduction}\label{sec:intro}

Topological complexity is a numerical homotopy invariant, defined by Michael Farber \cite{Far03,Far04} as part of his topological study of the motion planning problem from robotics. Given a path-connected space $X$, let $PX$ denote the space of all paths in $X$ endowed with the compact open topology, and let $\pi:PX\to X\times X$ denote the endpoint fibration given by $\pi(\gamma)=(\gamma(0),\gamma(1))$. Viewing $X$ as the configuration space of some mechanical system, one defines a \emph{motion planner} on a subset $A\subseteq X\times X$ to be a local section of $\pi$ on $A$, that is, a continuous map $\sigma: A\to PX$ such that $\pi\circ \sigma$ equals the inclusion of $A$ into $X\times X$. Assuming $X$ is an Euclidean Neighbourhood Retract (ENR), the topological complexity of $X$, denoted $\TC(X)$, is defined to be the smallest natural number $k$ such that $X\times X$ admits a partition into $k$ disjoint ENRs, each of which admits a motion planner.

Many basic properties of this invariant were established in the papers \cite{Far03,Far04}, which continue to inspire a great deal of research by homotopy theorists (a snapshot of the current state-of-the-art can be found in the conference proceedings volume \cite{GLV}). Here we simply mention that the topological complexity of spheres was calculated in \cite{Far03}; it is given by
\[
\TC(S^n) = \begin{cases} 2 & \mbox{if $n$ is odd,} \\ 3 & \mbox{if $n$ is even.} \end{cases}
\]

In the recent preprint \cite{Goubault}, Eric Goubault defined a variant of topological complexity for directed spaces. Recall that a directed space, or $d$-space, is a space $X$ together with a distinguished class of paths in $X$ called directed paths, satisfying certain axioms (full definitions will be given in Section 2). Partially ordered spaces give examples of $d$-spaces. The directed paths of a $d$-space form a subspace $\overrightarrow{P}X$ of $PX$. The endpoint fibration restricts to a map $\chi:\overrightarrow{P}X\to X\times X$, which is not surjective in general. Its image, denoted $\Gamma_X\subseteq X\times X$, is the set of $(x,y)\in X\times X$ such that there exists a directed path from $x$ to $y$. A \emph{directed motion planner} on a subset $A\subseteq \Gamma_X$ is defined to be a local section of $\chi$ on $A$. The directed topological complexity of the $d$-space $X$, denoted $\dTC(X)$, is the smallest natural number $k$ such that $\Gamma_X$ admits a partition into $k$ disjoint ENRs, each of which admits a directed motion planner.

As remarked in the introduction to \cite{Goubault}, the directed topological complexity seems more suited to studying the motion planning problem in the presence of control constraints on the movements of the various parts of the system. It was shown in \cite{Goubault} to be invariant under a suitable notion of directed homotopy equivalence, and a few simple examples were discussed. The article \cite{GFS} in the present volume provides further examples, and proves several properties of directed topological complexity, including a product formula. It remains to find general upper and lower bounds for this invariant, and to give further computations for familiar $d$-spaces.

The contribution of this short note is to compute the directed topological complexity of directed spheres. For each $n\ge1$ the directed sphere $\overrightarrow{S^n}$ is the directed space whose underlying topological space is the boundary $\partial I^{n+1}$ of the $(n+1)$-dimensional unit cube, and whose directed paths are those paths which are non-decreasing in every coordinate.

\begin{thm*}
The directed topological complexity of directed spheres is given by
\[
\dTC(\overrightarrow{S^n})=2\mbox{ for all $n\ge1$.}
\]
\end{thm*}

This theorem will be proved in Section 3 below by exhibiting a partition of $\Gamma_{\overrightarrow{S^n}}$ into $2$ disjoint ENRs with explicit motion planners.

The first author wishes to thank the University of Aberdeen for their hospitality during her stay at the Institute of Mathematics, where this work was carried out. Both authors wish to thank Eric Goubault for useful discussions and for sharing with them preliminary versions of his results, and the anonymous referees for valuable comments.

\section{Preliminaries}

\begin{defn}[{M.\ Grandis, \cite{Grandis}}]
A \emph{directed space} or \emph{$d$-space} is a pair $(X,\overrightarrow{P}X)$ consisting of a topological space $X$ and a subspace $\overrightarrow{P}X\subseteq PX$ of the path space of $X$ satisfying the following axioms:
\begin{itemize}
\item constant paths are in $\overrightarrow{P}X$;
\item $\overrightarrow{P}X$ is closed under pre-composition with non-decreasing continuous maps $r:[0,1]\to [0,1]$;
\item $\overrightarrow{P}X$ is closed under concatenation.
\end{itemize}
The paths in $\overrightarrow{P}X$ are called \emph{directed paths} or \emph{dipaths}, and the space $\dPX$ is called the \emph{dipath space}.
\end{defn}
Examples of $d$-spaces include partially ordered spaces (where dipaths in $\overrightarrow{P}X$ consist of continuous order-preserving maps $\gamma:([0,1],\leq)\to (X,\leq)$) and cubical sets. We can also view any topological space $X$ as a $d$-space by taking $\dPX=PX$. The dipath space $\dPX$ is usually omitted from the notation for a $d$-space $(X,\dPX)$.

\begin{defn}[{\cite{Goubault}}]
Given a $d$-space $X$, let
\[
\Gamma_X=\{(x,y)\in X\times X\mid \exists\gamma\in \dPX\mbox{ such that }\gamma(0)=x,\gamma(1)=y\}\subseteq X\times X.
\]
The \emph{dipath space map} is given by
\[
\chi:\dPX\to \Gamma_X,\qquad\chi(\gamma) = \big(\gamma(0),\gamma(1)\big).
\]
\end{defn}
That is, the dipath space map is obtained from the classical endpoint fibration $\pi:PX\to X\times X$ by restriction of domain and codomain.

\begin{defn}[{\cite{Goubault}}]
Given a $d$-space $X$, its \emph{directed topological complexity}, denoted $\dTC(X)$, is defined to be the smallest natural number $k$ such that there exists a partition $\Gamma_X=A_1\sqcup\cdots \sqcup A_k$ into disjoint ENRs, each of which admits a continuous map $\sigma_i:A_i\to \dPX$ such that $\chi\circ\sigma_i = {\rm incl}:A_i\hookrightarrow \Gamma_X$.
\end{defn}

\begin{rems}
The dipath space map is not a fibration, in general. One can easily imagine directed spaces $X$ for which the homotopy type of the fibre $\dPX(x,y)$ is not constant on the path components of $\Gamma_X$. Related to this is the fact that, unlike in the classical case of $\TC(X)$, the above definition does not coincide with the alternative definition using open (or closed) covers. Both of these remarks are due to E.\ Goubault.

Note that we are using the unreduced version of $\dTC$, as in the article \cite{Goubault}.
\end{rems}

A notion of \emph{dihomotopy equivalence} was defined in \cite[Definition 3]{Goubault}, and it was shown in \cite[Lemma 6]{Goubault} that if $X$ and $Y$ are dihomotopy equivalent $d$-spaces then $\dTC(X)=\dTC(Y)$. Furthermore, a notion of \emph{dicontractibility} for $d$-spaces was outlined in \cite[Definition 4]{Goubault}, and \cite[Theorem 1]{Goubault} asserts that a $d$-space $X$ that is contractible in the classical sense has $\dTC(X)=1$ if and only if $X$ is dicontractible (see also \cite[Theorem 1]{GFS}). Here we will only require the following weaker assertion. Let us call a $d$-space $(X,\dPX)$ \emph{loop-free} if for all $x\in X$, the fibre $\dPX(x,x)$ consists only of the constant path at $x$.

\begin{lem}\label{lemma}
Let $X$ be a loop-free $d$-space for which $\dTC(X)=1$. Then for all $(x,y)\in \Gamma_X$, the corresponding fibre $\dPX(x,y)$ of the dipath space map is contractible.
\end{lem}
\begin{proof}
We reproduce the relevant part of the proof of \cite[Theorem 1]{Goubault} (with thanks to the anonymous referees, who pointed out the necessity for an additional assumption such as loop-freeness to ensure continuity of the homotopy below). Suppose $\dTC(X)=1$, and let $\sigma:\Gamma_X\to \dPX$ be a global section of $\chi:\dPX\to \Gamma_X$. Given $(x,y)\in \Gamma_X$, let $f:\{\sigma(x,y)\}\to \dPX(x,y)$ and $g:\dPX(x,y)\to\{\sigma(x,y)\}$ denote the inclusion and constant maps, respectively. Clearly $g\circ f={\rm Id}_{\{\sigma(x,y)\}}$, so to prove the lemma it suffices to give a homotopy $H:\dPX(x,y)\times I\to \dPX(x,y)$ from $f\circ g$ to ${\rm Id}_{\dPX(x,y)}$. Such a homotopy is defined explicitly by setting
\[
H(\gamma,t)(s) = \begin{cases} \gamma(s) & \mbox{for }0\le s\le \frac{t}{2} \\ \sigma\left(\gamma\left(\frac{t}{2}\right),\gamma\left(1-\frac{t}{2}\right)\right)\left(\frac{s-\frac{t}{2}}{1-t}\right) & \mbox{for }\frac{t}{2} \le s \le 1-\frac{t}{2} \\
\gamma(s) & \mbox{for }1-\frac{t}{2} \le s \le 1 \end{cases}
\]
for all $t,s\in [0,1]$. The assumption that $X$ is loop-free ensures that $\sigma(x,x)$ is constant at $x$ for all $x\in X$, which guarantees continuity at $t=1$.
\end{proof}

\section{Directed topological complexity of directed spheres}

\begin{defn} Let $n\geq 1$ be a natural number. The \emph{directed $n$-sphere}, denoted $\overrightarrow{S^n}$, is the $d$-space whose underlying space is the boundary $\partial I^{n+1}$ of the unit cube $I^{n+1}=[0,1]^{n+1}\subseteq \R^{n+1}$, and whose dipaths are those paths which are non-decreasing in each coordinate.
\end{defn}

We now fix some notation, most of which is borrowed from \cite[Section 6]{FRGH}. A point $\mathbf{x}=(x_0,\ldots , x_n)\in\R^{n+1}$ will be denoted $\mathbf{x}=x_0\cdots x_n$ for brevity. We use $-$ to denote an arbitrary element of $(0,1)$. Therefore we may indicate an arbitrary point in $\partial I^{n+1}$ by a string $x_0\cdots x_n$ where each $x_i\in \{0,-,1\}$, and at least one $x_i\in \{0,1\}$. We let $[n]=\{0,1,\ldots , n\}$ denote the set which indexes the coordinates.

For example, $--0$ denotes an arbitrary point in the interior of the bottom ($z=0$) face of $\partial I^3$, while $-11$ denotes a point on the interior of the top ($z=1$), back ($y=1$) edge. The point $---$ is not in $\partial I^3$.

With these notations, if $(x_0\cdots x_n,y_0\cdots y_n)\in \Gamma_{\dsn}$ then $x_i\leq y_i$ for all $i\in [n]$, but the converse does not hold. For example, any pair of the form $(0--,1--)$ is not in $\Gamma_{\overrightarrow{S^2}}$.

We are now ready to prove our main result, restated here for convenience.

\begin{thm*}
The directed topological complexity of directed spheres is given by
\[
\dTC(\overrightarrow{S^n})=2\mbox{ for all $n\ge1$.}
\]
\end{thm*}

\begin{proof}
We first show that $\dTC(\dsn)>1$. Since $\dsn$ is loop-free, it suffices by Lemma \ref{lemma} to find $(\mathbf{x},\mathbf{y})\in \Gamma_{\dsn}$ such that $\overrightarrow{P}\overrightarrow{S^n}(\mathbf{x},\mathbf{y})$ is not contractible. Fix $x_2,\ldots, x_n\in (0,1)$. It is clear that
\[
\overrightarrow{P}\overrightarrow{S^n}(00x_2\cdots x_n,11x_2\cdots x_n)\cong\overrightarrow{P}\overrightarrow{S^1}(00,11),
\]
and that the latter space is disconnected (since not all dipaths from $00$ to $11$ are dihomotopic), see Figure \ref{TC1}. In particular, $\overrightarrow{P}\overrightarrow{S^n}(00x_2\cdots x_n,11x_2\cdots x_n)$ is not contractible, hence $\dTC(\dsn)>1$.

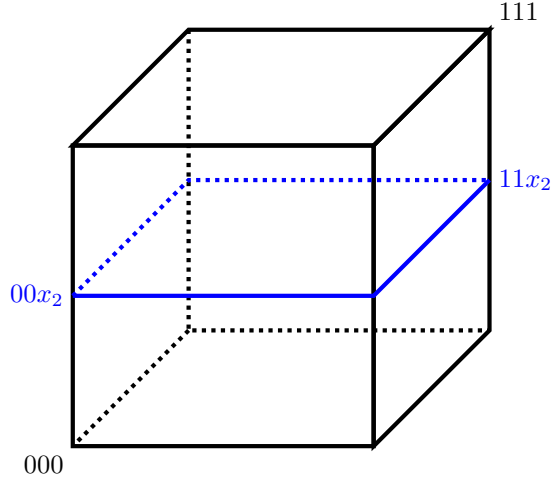
\begin{figure}
\begin{tikzpicture}
\pgfmathsetmacro{\cubex}{4}
\pgfmathsetmacro{\cubey}{4}
\pgfmathsetmacro{\cubez}{4}
\draw[ultra thick] (0,0,0) -- ++(-\cubex,0,0) -- ++(0,-\cubey,0) -- ++(\cubex,0,0) -- cycle;
\draw[ultra thick] (0,0,0) -- ++(0,0,-\cubez) -- ++(0,-\cubey,0) -- ++(0,0,\cubez) -- cycle;
\draw[ultra thick] (0,0,0) -- ++(-\cubex,0,0) -- ++(0,0,-\cubez) -- ++(\cubex,0,0) -- cycle;
\draw[ultra thick,dotted] (-\cubex,0,-\cubez) -- ++(0,-\cubey,0) -- ++(\cubex,0,0);
\draw[ultra thick,dotted] (-\cubex,-\cubey,-\cubez) -- ++(0,0,\cubez);
\draw[ultra thick,dotted,blue] (-\cubex,-.5*\cubey,0) -- ++(0,0,-\cubez) -- ++(\cubex,0,0);
\draw[ultra thick,blue] (-\cubex,-.5*\cubey,0) -- ++(\cubex,0,0) -- ++(0,0,-\cubez);
\node[below left] at (-\cubex,-\cubey,0){$000$};
\node[above right] at (0,0,-\cubez){$111$};
\node[left,blue] at (-\cubex,-.5*\cubey,0){$00x_2$};
\node[right,blue] at (0,-.5*\cubey,-\cubez){$11x_2$};
\end{tikzpicture}
\caption{Directed paths in $\protect\overrightarrow{S^2}$ from $00x_2$ to $11x_2$ must remain in the blue square, illustrating the homeomorphism $\protect\overrightarrow{P}\protect\overrightarrow{S^2}(00x_2,11x_2)\cong \protect\overrightarrow{P}\protect\overrightarrow{S^1}(00,11)$.}
\label{TC1}
\end{figure}

A more direct proof that $\dTC(\dsn)>1$, avoiding the use of Lemma \ref{lemma}, may be given as follows. Given $x,t\in (0,1)$, consider points $(\mathbf{x}_1,\mathbf{y}_1)=(t0x\cdots x,11x\cdots x)$ and $(\mathbf{x}_2,\mathbf{y}_2)=(0tx\cdots x,11x\cdots x)$, both in $\Gamma_{\dsn}$. A $d$-path from $\mathbf{x}_1$ to $\mathbf{y}_1$ is contained in the half-square $([0, 1]\times\{0\})\cup (\{1\}\times[0,1])\times\{x\}^{n-1}$, while a $d$-path from $\mathbf{x}_2$ to $\mathbf{y}_2$ is contained in the half-square $(\{0\}\times [0, 1])\cup ([0,1]\times\{1\})\times\{x\}^{n-1}$ (see Figure \ref{TC1} for the case $n = 2$, where the two half-squares are depicted in solid blue and dotted blue respectively). Assuming the existence of a continuous section for the dipath space map on all of $\Gamma_{\dsn}$ and letting $t$ tend to $0$ results in two disagreeing dipaths from $00x\cdots x$ to $11x\cdots x$, yielding a contradiction. (We are grateful to an anonymous referee for providing this argument.)

To prove that $\dTC(\dsn)\le2$, we will exhibit a partition $\Gamma_{\dsn}=A_1\sqcup A_2$ into two disjoint ENRs, each equipped with a continuous directed motion planner $\sigma_i:A_i\to \overrightarrow{P}(\dsn)$. 

Consider the $d$-space $\overrightarrow{\R^{n+1}}$, where the dipaths are non-decreasing in each coordinate. Here we have $\dTC(\overrightarrow{\R^{n+1}})=1$, for we can describe a directed motion planner $\widetilde{\sigma_1}$ on
\[
\Gamma_{\overrightarrow{\R^{n+1}}}=\{(x_0\cdots x_n,y_0\cdots y_n) \mid x_i\leq y_i\mbox{ for all }i\in[n]\}
\]
by first increasing $x_0$ to $y_0$, then increasing $x_1$ to $y_1$, and so on, finally increasing $x_n$ to $y_n$. It is not difficult to write a formula for $\widetilde{\sigma_1}$, and check that is it continuous. Similarly, we can define a second motion planner $\widetilde{\sigma_2}$ which first increases $x_n$ to $y_n$, then increases $x_{n-1}$ to $y_{n-1}$, and so on, finally increasing $x_0$ to $y_0$.

For $i=1,2$, let $B_i$ be the set of pairs $(\mathbf{x},\mathbf{y})$ in $\Gamma_{\dsn}\subseteq \Gamma_{\overrightarrow{\R^{n+1}}}$ such that the path $\widetilde{\sigma_i}(\mathbf{x},\mathbf{y})$ has image contained in $\partial I^{n+1}$. The restriction $\widetilde{\sigma_i}|_{B_i}:B_i\to \overrightarrow{P}(\dsn)$ is clearly continuous, and is a directed motion planner on $B_i$.

We will show that $B_1\cup B_2=\Gamma_{\dsn}$, and that both $B_1$ and its complement $U_1:=\Gamma_{\dsn}\setminus B_1$ are ENRs. Hence we may set $A_1=B_1$ and $A_2=U_1\subseteq B_2$ to obtain a cover by disjoint ENRs with motion planners $\sigma_i=\widetilde{\sigma_i}|_{A_i}$, as required.

The sets $B_1$ and $B_2$ are best understood in terms of their complements $U_1$ and $U_2:=\Gamma_{\dsn}\setminus B_2$, and in fact we will show that $U_1\cap U_2=\varnothing$.

Observe that $U_1$ is the set of pairs $(\mathbf{x},\mathbf{y})\in \Gamma_{\dsn}$ such that $\widetilde{\sigma_1}(\mathbf{x},\mathbf{y})$ enters the interior of the cube, and this can happen upon increasing any of the $n+1$ coordinates. Thus an element $(x_0\cdots x_n,y_0\cdots y_n)\in U_1$ satisfies the condition:
\begin{enumerate}
\item[$(U_1)$] There exists $j\in [n]$ such that:
\bi
\item $x_j<y_j$;
\item $y_i=-$ for all $i\in [n]$ with $i<j$;
\item $x_i=-$ for all $i\in [n]$ with $i>j$.
\ei
\end{enumerate}
Similarly, an element $(x_0\cdots x_n, y_0\cdots y_n)\in U_2$ satisfies the condition:
\begin{enumerate}
\item[$(U_2)$] There exists $k\in [n]$ such that:
\bi
\item $x_k<y_k$;
\item $x_i=-$ for all $i\in [n]$ with $i<k$;
\item $y_i=-$ for all $i\in [n]$ with $i>k$.
\ei
\end{enumerate}
Now suppose there is an element $(\mathbf{x},\mathbf{y})=(x_0\cdots x_n,y_0\cdots y_n)\in U_1\cap U_2$.

Then  there exist $j,k\in [n]$ such that
\begin{align*}
(\mathbf{x},\mathbf{y}) & = (x_0\cdots x_j-\cdots -,-\cdots - y_j\cdots y_n) \\
                        & = (-\cdots - x_k\cdots x_n,y_0\cdots y_k-\cdots - ).
\end{align*}
We observe that if $j<k$ then $\mathbf{x}=-\cdots -$, while if $j>k$ then $\mathbf{y}=-\cdots -$. Both of these give a contradiction, since $\mathbf{x}$ and $\mathbf{y}$ cannot be in the interior of the cube. Hence we must have $j=k$, and $(\mathbf{x},\mathbf{y})=(-\cdots -x_j-\cdots-,-\cdots -y_j-\cdots -)$ where $x_j<y_j$. Since $\mathbf{x}$ and $\mathbf{y}$ are in $\partial I^{n+1}$, we must have $x_j=0<1=y_j$. This gives $(\mathbf{x},\mathbf{y})=(-\cdots - 0-\cdots -,-\cdots - 1 -\cdots -)$ which is not in $\Gamma_{\dsn}$, a contradiction.

Thus $U_1\cap U_2=\varnothing$, and $B_1\cup B_2=\Gamma_{\dsn}$. It only remains to observe that both $B_1$ and $U_1$ are semi-algebraic subsets of $\R^{2n+2}$ (they are the solution sets of a finite number of linear equalities and inequalities) hence are ENRs.
\end{proof}


\begin{thebibliography}{99}

\bibitem{FRGH} L. Fajstrup, M. Raussen, E. Goubault, E. Haucourt, \emph{Components of the fundamental category}, Appl.\ Categ.\ Structures {\bf 12} (2004), no. 1, 81-–108.

\bibitem{Far03} M. Farber, \emph{Topological complexity of motion planning}, Discrete Comput.\ Geom. {\bf 29} (2003), no. 2, 211–-221.

\bibitem{Far04} M. Farber, \emph{Instabilities of robot motion}, Topology Appl. {\bf 140} (2004), no. 2-3, 245-–266.

\bibitem{Goubault} E. Goubault, \emph{On directed homotopy equivalences and a notion of directed topological complexity}, preprint. \href{https://arxiv.org/abs/1709.057027}{\texttt{arXiv:1709.057027}}

\bibitem{GFS} E. Goubault, M. Farber, A. Sagnier, \emph{Directed topological complexity}, preprint. \href{https://arxiv.org/abs/1812.09382}{\texttt{arXiv:1812.09382}}

\bibitem{Grandis} M. Grandis, \emph{Directed homotopy theory. I.}, Cah.\ Topol.\ G\'{e}om.\ Diff\'{e}r. Cat\'{e}g. {\bf 44} (2003), no. 4, 281-–316.

\bibitem{GLV} M. Grant, G. Lupton and L. Vandembroucq (eds), \emph{Topological complexity and related topics}, Contemp.\ Math. {\bf 702}, American Mathematical Society, 2018.

\end{thebibliography}
\end{document}